\documentclass[11pt]{article}
\usepackage{amsmath,amsthm,amsfonts,amssymb,amscd, amsxtra}

\newcommand{\rank}{rank}
\newtheorem{theorem}{Theorem}
\newtheorem{lemma}[theorem]{Lemma}

\newtheorem{corollary}[theorem]{Corollary}
\newtheorem{proposition}[theorem]{Proposition}

\newtheorem{remark}{Remark}
\begin{document}
\title{Local convergence analysis of Gauss-Newton's   method\\ under  majorant condition}

\author{ O. P. Ferreira\thanks{IME/UFG, Campus II- Caixa
    Postal 131, CEP 74001-970 - Goi\^ania, GO, Brazil (E-mail:{\tt
      orizon@mat.ufg.br}). The author was supported in part by
    CNPq Grant 473756/2009-9, CNPq Grant 302024/2008-5, PRONEX--Optimization(FAPERJ/CNPq) and FUNAPE/UFG.}
   \and  M. L. N. Gon\c calves \thanks{COPPE-Sistemas, Universidade Federal do Rio de Janeiro, Rio de Janeiro, RJ 21945-970, BR (E-mail:{\tt
      maxlng@cos.ufrj.br}). The author was supported in part by
    CNPq Grant 473756/2009-9.} \and P. R. Oliveira \thanks{COPPE-Sistemas, Universidade Federal do Rio de Janeiro,
Rio de Janeiro, RJ 21945-970, BR (Email: {\tt poliveir@cos.ufrj.br}).
This author was supported in part by CNPq.} }
 \maketitle
\begin{abstract}
The Gauss-Newton's method for solving  nonlinear least squares problems is studied in this paper. Under the hypothesis that the derivative of the function associated with the least square problem satisfies a majorant condition, a local convergence analysis is presented.
This analysis allow us to obtain the optimal  convergence radius, the
biggest range for the uniqueness of solution, and  to unify two  previous and unrelated results.
\end{abstract}


\maketitle
\section{Introduction}\label{sec:int}
The Gauss-Newton's method  is one  of the most efficient methods
known for solving  nonlinear least squares problems
\begin{equation}\label{eq:p1}
\min \;\frac{1}{2}F(x)^TF(x),
\end{equation}
where  $F:\Omega\to \mathbb{R}^{m}$ is  differentiable
function, $\Omega \subset \mathbb{R}^{n}$ is an open set and
$m\geq n$. Formally, the Gauss-Newton's method is described as follows: Given
a initial point $x_0 \in \Omega$, define
$$
x_{k+1}={x_k}- \left[F'(x_k)^TF'(x_k)\right]^{-1}F'(x_k)^TF(x_k),
\qquad k=0,1,\ldots.
$$
The convergence of this method may fail or it even fail to be well defined. To ensure that the method is well defined and converges
to a solution of \eqref{eq:p1}, some conditions must be impose. For
instance, the classical convergence analysis (see \cite{DN1, NC1}) requires that $F'$ satisfies the  Lipschitz
condition and the initial iterate to be "close enough" the
solution, but it cannot make us clearly see how big is the convergence radius of
the ball.

In the last years, there are many papers dealing with the convergence of the Newton's methods, including the
Gauss-Newton's method, by relaxing the assumption of Lipschitz
continuity of the derivative (see
\cite{C11,C10,MAX1,F08,FS06,LZJ,XW10,Y84,AAA,MR2475307,Proinov20103,Li2010}). Those works
in addition to improving the convergence theory (this allows us
estimate the convergence radius and to enlarge the range of
application)  also permit us unify two results.

Our aim in this paper is  present a new local convergence
analysis for Gauss-Newton's method under majorant condition
introduced by Kantorovich \cite{KAN1},  and used with successful  by Ferreira and Gon\c calves
\cite{MAX1}, Ferreira \cite{F08} and Ferreira and
Svaiter \cite{FS06} for studying Newton's method. In our analysis, the classical Lipschitz condition
is relaxed using a majorant function. It is worth pointing out
that this condition is equivalent to Wang's condition introduced
in \cite{XW10} and used by Chen,
Li \cite{C11, C10} and Li, et al. \cite{LZJ}  for studying Gauss-Newton and Newton's method.  The convergence analysis
presented provides a clear relationship between the majorant
function, which relax the Lipschitz continuity of the derivate, and the function associated
with the nonlinear least square problem, see for example Lemmas \ref{wdns}, \ref{pr:taylor} and \ref{l:wdef}. Thus, the results presented here has the conditions and proof of convergence simpler and more didactic. Also, as in Chen, Li \cite{C11},
it allow us to obtain the biggest range for the uniquess of
solution and the optimal convergence radius  for the method with
respect to majorant function.  Moreover,  two unrelated
previous results pertaining Gauss-Newton's method  are unified.

The organization of the paper is as follows. In Sect.
\ref{sec:int.1}, we list some notations and basic results used in
our presentation. In Sect. \ref{lkant} the main result is stated,
and in  Sect. \ref{sec:PMF}  some properties involving the
majorant function are established. In Sect. \ref{sec:MFNLO} we
presented the relationships between the majorant function and the
non-linear function $F$, and in Sect. \ref{sec:unique} the optimal ball of convergence and the uniqueness of solution of convergence are
established. In Sect. \ref{sec:proof} the main result is proved and
some applications of this result are given in Sect. \ref{sec:ec}.
\subsection{Notation and auxiliary results} \label{sec:int.1}
The following notations and results are used throughout our
presentation.  The open and closed ball
at $a \in \mathbb{R}^n$ and radius $\delta>0$ are denoted, respectively by
$$
B(a,\delta) =\{ x\in \mathbb{R}^n ;\; \|x-a\|<\delta \}, \qquad B[a,\delta] =\{ x\in \mathbb{R}^n ;\; \|x-a\|\leqslant \delta \}.
$$
Let $ \mathbb{R}^{m \times n}$ denote the set of all
$ m \times n$ matrix $A$, $A^\dagger$
 denote the Moore-Penrose inverse of matrix $A$, and if $A$
 has full rank(namely: rank(A)=min(m,n)=n) then $A^\dagger=(A^TA)^{-1}A^T.$

\begin{lemma}(Banach's Lemma) \label{lem:ban1}
Let $B \in \mathbb{R}^{m\times m}$   and $I\in\mathbb{R}^{m \times m}$, the identity operator.
If  $\|B-I\|<1$,  then $B$ is invertible and  $ \|B^{-1}\|\leq
1/\left(1- \|B-I\|\right). $
\end{lemma}
\begin{proof}  See the proof of Lemma 1, p. 189 of Smale \cite{S86}  with  $A=I$  and  $c=\|B-I\|$.
\end{proof}
\begin{lemma} \label{lem:ban2}
Suppose that $ A,E \in \mathbb{R}^{m \times n}(m\geq n)$, $B=A+E,$
$\|EA^\dagger \|<1$, rank(A)=n, then rank(B)=n.
\end{lemma}
\begin{proof} In fact, $B=A+E=(I+EA^\dagger)A,$ from the condition $\|EA^\dagger\|<1,$ we have of Lemma \ref{lem:ban1} that $I+EA^\dagger$ is invertivel. So rank(B)=rank(A)=n.
\end{proof}
\begin{lemma} \label{lem:ban}
Suppose that $ A,E \in \mathbb{R}^{m \times n}$, $B=A+E,$
$\|A^\dagger\|\|E\|<1,$ $rank (A)=rank(B)$, then
$$
\|B^\dagger\|\leq \frac{\|A^\dagger\|}{ 1- \|A^\dagger\|\|E\|}.
$$
Moreover,  if $rank(A)=rank(B)=min(m,n)$, there holds
$$
\|B^\dagger-A^\dagger\|\leq \frac{\sqrt{2}\|A^\dagger\|^2\|E\|}{
1- \|A^\dagger\|\|E\|}.
$$
\end{lemma}
\begin{proof} See  Lema 5.1. on pp. 40  of Stewart \cite{G1} and  Wedin
\cite{W1}.
\end{proof}
\begin{proposition} \label{le:ess}
If $0\leq t <1$, then $ \sum
_{i=0}^{\infty}(i+2)(i+1)t^{i}=2/(1-t)^3. $
\end{proposition}
\begin{proof} Take $k=2$ in  Lemma 3, pp. 161 of  Blum, et al. \cite{BCSS97}.
\end{proof}
Also, the following auxiliary results of elementary convex analysis
will be needed:
\begin{proposition}
  \label{pr:conv.aux2}
Let $R>0$. If $\varphi:[0, R)\to \mathbb{R}$ is  convex, then
$$
D^+ \varphi(0)={\lim}_{u\to 0+} \; \frac{\varphi(u)-\varphi(0)}{u}
={\inf}_{0<u} \;\frac{\varphi(u)-\varphi(0)}{u}. \\
$$
\end{proposition}
\begin{proof} See  Theorem 4.1.1  on pp. 21 of Hiriart-Urruty and Lemar\'echal \cite{HL93}.
\end{proof}
\begin{proposition}\label{pr:conv.aux1}
Let $\epsilon>0$  and  $\tau \in [0,1]$. If $\varphi:[0,\epsilon)
\rightarrow\mathbb{R}$ is convex, then $l:(0,\epsilon) \to
\mathbb{R}$ define by
$$
l(t)=\frac{\varphi(t)-\varphi(\tau t)}{t},
$$
is increasing.
\end{proposition}
\begin{proof} See  Theorem 4.1.1 and Remark 4.1.2 on pp. 21 of Hiriart-Urruty and Lemar\'echal \cite{HL93}.
\end{proof}

\section{Local analysis  for Gauss-Newton's method } \label{lkant}
Our goal is to state and prove a local theorem for Gauss-Newton's
method. First, we will
 prove some results regarding the scalar majorant function, which relaxes the Lipschitz condition of the derivative  of the function associated with the nonlinear least square problem. Then we will  show that Gauss-Newton's method is well-defined and  converges.
 We will also prove the uniqueness of the solution in a suitable region and the convergence rate will be established.
 The statement of the theorem is as follows:
\begin{theorem}\label{th:nt}
Let $\Omega\subseteq \mathbb{R}^{n}$ be an open set,
$F:{\Omega}\to \mathbb{R}^{m}$ a continuously differentiable
function and $m\geq n$. Let $x_* \in \Omega,$ $R>0$ and
$$
c:=\|F(x_*)\|, \qquad \beta:=\left\|[F'(x_*)^TF'(x_*)]^{-1}F'(x_*)^T\right\|, \qquad \kappa:=\sup \left\{ t\in [0, R): B(x_*, t)\subset\Omega \right\}.
$$
Suppose that $x_*$ is a solution of  \eqref{eq:p1},
$F '(x_*)$ has full rank
and there exists a $f:[0,\; R)\to \mathbb{R}$
continuously differentiable such that
  \begin{equation}\label{Hyp:MH}
\left\|F'(x)-F'(x_*+\tau(x-x_*))\right\| \leq
f'\left(\|x-x_*\|\right)-f'\left(\tau\|x-x_*\|\right),
  \end{equation}
  for  all $\tau \in [0,1]$,  $x\in B(x_*, \kappa)$  and
\begin{itemize}
  \item[{\bf h1)}]  $f(0)=0$ and $f'(0)=-1$;
  \item[{\bf  h2)}]  $f'$ is convex and  strictly increasing;
   \item[{\bf  h3)}]  $\sqrt{2} \,c\, \beta ^2 D^+ f'(0)<1$.
\end{itemize}
Let be given the positive constants $\nu  :=\sup \left\{t \in[0, R):\beta[f'(t)+1]<1\right\},$
\[
 \rho :=\sup \left\{t \in(0, \nu):\frac{\beta[tf'(t)-f(t)]+\sqrt{2}c
 \beta^2[f'(t)+1]}{t[1-\beta(f'(t)+1)]}<1\right\}, \quad
    r :=\min  \left\{\kappa, \, \rho \right\}.
\]
Then, the Gauss-Newton's method for solving \eqref{eq:p1}, with
starting point $x_0\in B(x_*, r)/\{x_*\}$
\begin{equation} \label{eq:DNS}
x_{k+1}={x_k}-
\left[F'(x_k)^TF'(x_k)\right]^{-1}F'(x_k)^TF(x_k),\qquad
k=0,1\ldots,
\end{equation}
 is well defined, the  generated  sequence $\{x_k\}$ is contained in $B(x_*,r)$, converges to $x_*$ and
\begin{multline} \label{eq:q2}
    \|x_{k+1}-x_*\| \leq
    \frac{\beta[f'(\|x_0-x_*\|)\|x_0-x_*\|-f(\|x_0-x_*\|)]}{\|x_0-x_*\|^2[1-\beta(f'(\|x_0-x_*\|)+1)]}{\|x_k-x_*\|}^{2}\\+
    \frac{\sqrt{2}c    \beta^2[f'(\|x_0-x_*\|)+1]}{\|x_0-x_*\|[1-\beta(f'(\|x_0-x_*\|)+1)]}\|x_k-x_*\|,\qquad
    k=0,1,\ldots.
\end{multline}
Moreover, if $[\beta(\rho
f'(\rho)-f(\rho))+\sqrt{2}c\beta^2(f'(\rho)+1)]/[\rho(1-\beta(f'(\rho)+1))]=1$
and $\rho<\kappa$, then $r=\rho$ is the best  possible convergence
radius.\\
If, additionally,
\begin{itemize}
 \item[{\bf  h4)}]  $ 2\,c\,\beta_0\, D^+f'(0)<1$, then the point $x_*$ is the unique solution of \eqref{eq:p1} in
$B(x_*, \sigma)$, where
\end{itemize}
$
0<\sigma:=\sup\{t\in(0, \kappa):[ \beta(f(t)/t+1)+c\beta_0(f'(t)+1)/t]<
 1\}, \qquad \beta_0:=\|[F'(x_*)^TF'(x_*)]^{-1}\|.
$
\end{theorem}
\begin{remark}
The inequality \eqref{eq:q2} shows that if $c=0$ (the so-called
zero-residual case), then the Gauss-Newton's method is locally
$Q$-quadratically convergent to $x_*$. This behavior  is quite
similar to that of Newton's method (see \cite{F08, XW10}). If $c$ is small relative (the so-called
small-residual case), the inequality \eqref{eq:q2} implies that the
Gauss-Newton's method is locally $Q$-linearly convergent to $x_*$.
However, if $c$ is large (the so-called large-residual case), the
Gauss-Newton's method may not be locally convergent at all, see
condition {\bf h3} and also  example 10.2.4 on pp.225 of \cite{DN1}. Hence, we may conclude that the Gauss-Newton's
method perform better on zero-or small-residual problems than on
large-residual problems, while the Newton's method is equally
effective in all these cases.
\end{remark}
For the zero-residual problems, i.e., $c=0$, the  Theorem \ref{th:nt}  becomes:
\begin{corollary} \label{col:pc1}
Let $\Omega\subseteq \mathbb{R}^{n}$ be an open set,
$F:{\Omega}\to \mathbb{R}^{m}$ a continuously differentiable
function and $m\geq n$. Let $x_* \in \Omega,$ $R>0$ and
$$\beta:=\left\|[F'(x_*)^TF'(x_*)]^{-1}F'(x_*)^T\right\|,
\qquad \kappa:=\sup \left\{ t\in [0, R): B(x_*, t)\subset\Omega \right\}.
$$Suppose that $F(x_*)=0$, $F '(x_*)$ has full rank and there exists
a $f:[0,\; R)\to \mathbb{R}$  continuously differentiable such
that
  $$
\left\|F'(x)-F'(x_*+\tau(x-x_*))\right\| \leq
f'\left(\|x-x_*\|\right)-f'\left(\tau\|x-x_*\|\right),
  $$
  for all  $\tau \in [0,1],$  $x\in B(x_*, \kappa)$ and
\begin{itemize}
  \item[{\bf h1)}]  $f(0)=0$ and $f'(0)=-1$;
  \item[{\bf  h2)}]  $f'$ is convex and  strictly increasing.
  \end{itemize}
Let be given the positive constants $\nu=:\sup\{t \in[0, \nu):\beta[f'(t)+1]<1\},$
 $$\rho:=\sup\{t \in(0, \nu):[\beta(tf'(t)-f(t))]/[t(1-\beta(f'(t)+1))]<1\}, \qquad
  r:=\min \left\{\kappa, \, \rho\right\}.
$$
Then, the Gauss-Newton's method for solving \eqref{eq:p1}, with
initial point $x_0\in B(x_*, r)/\{x_*\}$
$$
x_{k+1}={x_k}-
\left[F'(x_k)^TF'(x_k)\right]^{-1}F'(x_k)^TF(x_k),\qquad
k=0,1\ldots,
$$
 is well defined,  the sequence generated $\{x_k\}$ is contained in $B(x_*,r)$ and  converges to
 $x_*$ which is the unique solution of \eqref{eq:p1} in
$B(x_*, \sigma)$, where $ 0<\sigma:=\sup\{0<t<\kappa:\beta[f(t)/t+1]<1
 \}.$ Moreover, there holds
$$
    \|x_{k+1}-x_*\| \leq
    \frac{\beta[f'(\|x_0-x_*\|)\|x_0-x_*\|-f(\|x_0-x_*\|)]}{\|x_0-x_*\|^2[1-\beta(f'(\|x_0-x_*\|)+1)]}{\|x_k-x_*\|}^{2}\\,\qquad
    k=0,1,\ldots.
$$
If, additionally, $[\beta(\rho
f'(\rho)-f(\rho))]/[\rho(1-\beta(f'(\rho)+1))]=1$ and
$\rho<\kappa$, then $r=\rho$ is the best  possible convergence
radius.
\end{corollary}
\begin{remark}
When $m=n,$  the Corollary \ref{col:pc1}  is similar to the
result on  Newton's method for solving nonlinear equations
$F(x)=0$, which has been obtained  by Ferreira
\cite{F08} in Theorem 2.1.
\end{remark}
In order to prove Theorem \ref{th:nt} we need some results. From
here on, we assume  that all assumptions of Theorem \ref{th:nt}
hold.
\subsection{The majorant function } \label{sec:PMF}
Our first goal is to show that the constant $\kappa$ associated
with $\Omega$ and the constants  $\nu$,  $\rho$ and $\sigma$
associated with the majorant function $f$ are positive. Also, we
will prove some results related to the function $f$.

We begin by noting that $\kappa>0$, because $\Omega$ is an open
set and $x_*\in \Omega$.
\begin{proposition}  \label{pr:incr1}
The constant $  \nu $ is positive and  and there holds
$$
\beta[f'(t)+1]<1, \qquad t\in (0, \nu).
$$
\end{proposition}
\begin{proof}
 As   $f'$ is continuous in $(0,R)$ and $f'(0)=-1,$ it is easy to conclude that
$$\lim _{t \to 0}\beta[f'(t)+1]=0. $$
Thus, there exists a $\delta>0$ such that $\beta(f'(t)+1)<1$ for
all $t\in (0, \delta)$. Hence,  $\nu>0.$

Using  {\bf  h2} and definition of $\nu$ the last part of the proposition follows.

\end{proof}
\begin{proposition} \label{pr:incr101}
The following functions are  increasing:

\begin{itemize}
 \item[{\bf i)}] $[0,\, R) \ni t \mapsto 1/[1-\beta(f'(t)+1)];$
 \item[{\bf ii)}]  $(0,\, R) \ni t \mapsto [tf'(t)-f(t)]/t^2;$
\item[{\bf iii)}]  $(0,\, R) \ni t \mapsto [f'(t)+1]/t;$
\item[{\bf iv)}]  $(0,\, R) \ni t \mapsto f(t)/t.$
\end{itemize}
As a consequence, are increasing the following functions
$$
(0,\, R)\ni t\mapsto \frac{tf'(t)-f(t)}{t^2[1-\beta(f'(t)+1)]},
\qquad  \qquad (0,\, R)\ni t\mapsto
\frac{f'(t)+1}{t[1-\beta(f'(t)+1)]}.
$$
\end{proposition}
\begin{proof}
 The item~{\bf i} is immediate, because  $f'$ is strictly increasing in $[0, R)$.

 For proving item~{\bf ii}, note that after  some  simple algebraic manipulations we have
  $$
  \frac{tf'(t)-f(t)}{t^2}=\int_{0}^{1}\frac{f'(t)-f'(\tau t)}{t}\; d\tau.
  $$
  So, applying
Proposition~\ref{pr:conv.aux1} with $f'=\varphi$ and $\epsilon=R$ the statement follows.

For establishing  item~{\bf iii} use  $\bf h2$, $f'(0)=-1$ and
Proposition~\ref{pr:conv.aux1} with $f'=\varphi,$ $\epsilon=R$ and
$\tau =0.$

Assumption $\bf h2$ implies that $f$ is convex. As $f(0)=0,$ we have  $f(t)/t=[f(t)-f(0)]/[t-0]$. Hence, item~{\bf iv} follows by applying Proposition~\ref{pr:conv.aux1} with  $f=\varphi$ and and
$\tau =0.$

For proving that the  functions in the last part are increasing combine item {\bf i} with {\bf ii} for the first function  and  {\bf i} with {\bf iii} for the second function.
\end{proof}
\begin{proposition}\label{pr:incr102}
The constant $ \rho $ is  positive  and there holds
$$
\frac{\beta[tf'(t)-f(t)]+\sqrt{2}c
\beta^2[f'(t)+1]}{t[1-\beta(f'(t)+1)]}<1, \qquad \forall \; t\in
(0, \, \rho).
$$
\end{proposition}
\begin{proof}
First, using $\bf h1$ and  some algebraic manipulation  gives
$$
\frac{\beta[tf'(t)-f(t)]+\sqrt{2}c
\beta^2[f'(t)+1]}{t[1-\beta(f'(t)+1)]} =\frac{\beta\left[f'(t)-\displaystyle\frac{f(t)-f(0)}{t-0}\right]+
\sqrt{2}\,c \,\beta^2\,\displaystyle\frac{f'(t)-f'(0)}{t-0}}{1-\beta(f'(t)+1)}.
$$
Combing last equation with the assumption that $f'$ is convex, we obtain from Proposition~\ref{pr:conv.aux2} that
$$
\lim_{t \to 0}\frac{\beta[tf'(t)-f(t)]+\sqrt{2}c
\beta^2[f'(t)+1]}{t[1-\beta(f'(t)+1)]}=\sqrt{2}c \beta^2D^+f'(0).
$$
Now, using {\bf h3}, i.e.,  $\sqrt{2}c \beta^2D^+f'(0)<1$, we conclude
that there exists a $\delta>0$ such that
$$
\frac{\beta[tf'(t)-f(t)]+\sqrt{2}c
\beta^2[f'(t)+1]}{t[1-\beta(f'(t)+1)]}<1, \qquad
 t\in
(0, \delta),
$$
Hence,  $\delta\leq \rho$, which  prove  the first
statement.

For concluding the proof, we use the definition of $\rho$,
above inequality  and  last part of Proposition~\ref{pr:incr101}.
\end{proof}

\begin{proposition}  \label{pr:sig}
The constant  $\sigma$ is positives and there holds
$$
\beta(f(t)/t+1)+c\beta_0(f'(t)+1)/t<1, \qquad t\in (0, \sigma).
$$
\end{proposition}
\begin{proof}
For proving that
$\sigma>0$  we need the  assumption {\bf h4}. First,  note that condition  {\bf h1} implies
$$
\beta\left[\frac{f(t)}{t}+1\right]+c\beta_0\frac{f'(t)+1}{t}= \beta\left[\frac{f(t)-f(0)}{t-0}-f'(0)\right]+ c\beta_0\frac{f'(t)-f'(0)}{t-0}.
$$
Therefore, using last equality together with the assumption that $f'$ is convex and {\bf h4} we have $\lim _{t \to 0}[ \beta(f(t)/t+1)+c\beta_0(f'(t)+1)/t]=c\beta_0D^+f'(0)<1/2$. Thus, there exists a  $\delta>0$
such that
$$
 \beta\left[\frac{f(t)}{t}+1\right]+ c\beta_0\frac{f'(t)+1}{t}<~1, \qquad t\in (0, \delta).
$$
Hence,  $\delta\leq \sigma$, which   prove the first
statement.

For concluding the proof, we use the definition of $\sigma$, above
inequality \ and  items {\bf iii} and {\bf iv}  in Proposition~\ref{pr:incr101}.
\end{proof}
\subsection{Relationship of the majorant function with the non-linear function} \label{sec:MFNLO}
In this section we will present the main  relationships between
the majorant function $f$ and the  function $F$ associated with the nonlinear least square problem.
\begin{lemma} \label{wdns}
Let $x \in \Omega$. If \,$\| x-x_*\|<\min\{\nu,\kappa\}$, then
$F'(x)^T F'(x) $ is invertible and the following
inequalities hold
$$
\left\|[F'(x)^{T}F'(x)]^{-1}F'(x)^{T}\right\|\leq \frac{\beta}{1-\beta [f'(\|
x-x_*\|)+1]},
$$
and
 $$
 \left\|[F'(x)^{T}F'(x)\big]^{-1}F'(x)^{T}-[F'(x_*)^{T}F'(x_*)]^{-1}F'(x_*)^{T}\right\|
<
\frac{\sqrt{2}\beta^2[f'(\|x-x_*\|)+1]}{1-\beta[f'(\|x-x_*\|)+1]}.
 $$
In particular, $F'(x)^T F'(x)$ is invertible in $B(x_*, r)$.
\end{lemma}
\begin{proof}
Let $x \in \Omega$ such that \,$\| x-x_*\|<\min\{\nu,\kappa\}$. Since $\| x-x_*\|<\nu$,  using the definition of $\beta$, the inequality \eqref{Hyp:MH} and last part of Proposition~\ref{pr:incr1} we have
$$
\|F'(x)-F'(x_*)\|\|[F'(x_*)^TF'(x_*)]^{-1}F'(x_*)^T\|\leq
\beta[f'(\| x-x_*\|)-f'(0)] < 1.
$$
For simply the notations define the following matrices
\begin{equation} \label{eq:daux}
A=F'(x_*), \qquad  B=F'(x), \qquad E=F'(x)-F'(x_*).
\end{equation}
The last definitions together with latter inequality imply that
$$
\|EA^\dagger\|\leq\|E\|\|A^\dagger\| < 1,
$$
which, using that $F'(x_*)$  has full rank, implies in view of Lemma \ref{lem:ban2}
that  $F'(x)$ has full rank. So, $F'(x)^T F'(x) $ is invertible and by definition of
 $r$ we obtain that $F'(x)^T F'(x)$ is invertible for all $x\in B(x_*, r)$.

We already knows that $\rank F'(x)=\rank F'(x_*)=n$. Hence, for concluding the lemma,
first use  definitions in \eqref{eq:daux}  to obtain that $\rank (B)=\rank(A)=n$ and then combine  the above inequality and Lemma \ref{lem:ban}.
\end{proof}

Now, it is convenient to study the linearization error of $F$ at point in~$\Omega$,  for that  we define
\begin{equation}\label{eq:def.er}
  E_F(x,y):= F(y)-\left[ F(x)+F'(x)(y-x)\right],\qquad y,\, x\in \Omega.
\end{equation}
We will bound this error by the error in the linearization on the
majorant function $f$
\begin{equation}\label{eq:def.erf}
        e_f(t,u):= f(u)-\left[ f(t)+f'(t)(u-t)\right],\qquad t,\,u \in [0,R).
\end{equation}
\begin{lemma}  \label{pr:taylor}
If  $\|x-x_*\|< \kappa$, then  there holds $ \|E_F(x, x_*)\|\leq
e_f(\|x-x_*\|, 0). $
\end{lemma}
\begin{proof}
 Since   $B(x_*, \kappa)$ is convex,  we obtain that $x_*+\tau(x-x_*)\in B(x_*, \kappa)$, for $0\leq \tau \leq 1$.
 Thus,  as $F$ is  continuously differentiable in $\Omega$, definition of $E_F$ and some simple manipulations yield
$$
\|E_F(x,x_*)\|\leq \int_0 ^1 \left \|
[F'(x)-F'(x_*+\tau(x-x_*))]\right\|\,\left\|x_*-x\right\| \;
d\tau.
$$
From  the last inequality  and the assumption \eqref{Hyp:MH}, we
obtain
$$
\|E_F(x,x_*)\| \leq \int_0 ^1
\left[f'\left(\left\|x-x_*\right\|\right)-f'\left(\tau\|x-x_*\|\right)\right]\|x-x_*\|\;d\tau.
$$
Evaluating the above integral and using definition of $e_f$, the
statement follows.
\end{proof}
Lemma \ref{wdns} guarantees, in particular,  that  $F'(x)^TF'(x)$
is invertible in $B(x_*, r)$ and consequently, the Gauss-Newton
iteration map is well-defined.  Let us call $G_{F}$, the
Gauss-Newton iteration map for $F$ in that region:
\begin{equation} \label{NF}
  \begin{array}{rcl}
  G_{F}:B(x_*, r) &\to& \mathbb{R}^n\\
    x&\mapsto& x- \left[F'(x)^TF'(x)\right]^{-1}F'(x)^TF(x).
  \end{array}
\end{equation}
One can apply a \emph{single} Gauss-Newton iteration on any $x\in
B(x_*, r)$ to obtain $G_{F}(x)$ which may not belong to $B(x_*,
r)$, or even may not belong to the domain of $F$. So, this is
enough to guarantee well definedness of only one iteration.   To
ensure that Gauss-Newton iterations may be repeated indefinitely,
we need following result.
\begin{lemma} \label{l:wdef}
Let $x \in \Omega$. If \,$\| x-x_*\|<r$,  then $G_F$ is well
defined and there holds
\begin{multline*}
\|G_F(x)-x_{*}\|\leq
\frac{\beta[f'(\|x-x_*\|)\|x-x_*\|-f(\|x-x_*\|)]}{\|x-x_*\|^2[1-\beta(f'(\|x-x_*\|)+1)]}\|x-x_*\|^2\\
+\frac{\sqrt{2}c
\beta^2[f'(\|x-x_*\|)+1]}{\|x-x_*\|[1-\beta(f'(\|x-x_*\|)+1)]}\|x-x_*\|.
\end{multline*}
In particular,
$$
\|G_{F}(x)-x_{*}\|< \|x-x_*\|.
$$
\end{lemma}
\begin{proof}
First note that, as $\|x-x_*\|<r$ it follows from  Lemma
\ref{wdns} that $F'(x)^T F'(x)$ is invertible, then $G_F(x)$ is
well defined. Since $F'(x_*)^{T}F(x_*)=0$, some
algebraic manipulation and \eqref{NF} yield
\begin{multline*}
G_F(x)-x_{*}= \big[F'(x)^{T}F'(x)\big]^{-1}F'(x)^{T}[F'(x)(x-x_*)-F(x)+F(x_*)]\\
+\big[F'(x_*)^{T}F'(x_*)\big]^{-1}F'(x_*)^{T}F(x_*)-\big[F'(x)^{T}F'(x)\big]^{-1}F'(x)^{T}F(x_*).
\end{multline*}
From the last equation,  properties of the norm and
\eqref{eq:def.er}, we obtain
\begin{multline*}
\|G_F(x)-x_{*}\|\leq \left\|[F'(x)^{T}F'(x)]^{-1}F'(x)^{T}\right\|\left\|E_{F}(x,x_*)\right\|\\
+\left\|[F'(x_*)^{T}F'(x_*)]^{-1}F'(x_*)^{T}-[F'(x)^{T}F'(x)\big]^{-1}F'(x)^{T}\right\| \left\|F(x_*)\right\|.
\end{multline*}
Since $c=\|F(x_*)\|$, combining last inequality with Lemmas
\ref{wdns} and \ref{pr:taylor} we have
\[
\|G_F(x)-x_{*}\|\leq \frac{\beta e_f(\|x-x_*\|,
0)}{1-\beta(f'(\|x-x_*\|)+1)}+\frac{\sqrt{2}c
\beta^2(f'(\|x-x_*\|)+1)}{1-\beta(f'(\|x-x_*\|)+1)}.
\]
Now, using  \eqref{eq:def.erf} and {\bf h1}, we conclude from last
inequality that
\[
\|G_F(x)-x_{*}\|\leq
\frac{\beta[f'(\|x-x_*\|)\|x-x_*\|-f(\|x-x_*\|)]}{1-\beta(f'(\|x-x_*\|)+1)}+\frac{\sqrt{2}c
\beta^2[f'(\|x-x_*\|)+1]}{1-\beta(f'(\|x-x_*\|)+1)},
\]
which is equivalent to the first inequality of the lemma.

To end the proof first note that the right hand side of the first inequality of the lemma is equivalent to
$$
\left[
\frac{\beta[f'(\|x-x_*\|)\|x-x_*\|-f(\|x-x_*\|)]}{\|x-x_*\|[1-\beta(f'(\|x-x_*\|)+1)]}
+\frac{\sqrt{2}c \beta^2[f'(\|x-x_*\|)+1]}{\|x-x_*\|[1-\beta(f'(\|x-x_*\|)+1)]} \right]\|x-x_*\|.
$$
On the other hand, as $x\in B(x_*,r)/\{x_*\}$, i.e., $0<\|x-x_*\|<r\leq \rho$ we apply the Proposition~\ref{pr:incr102} with $t=\|x-x_*\|$ to conclude that the quantity in the bracket above is less than one. So, the last inequality of the lemma  follows.
\end{proof}
\subsection{Optimal ball of convergence and uniqueness} \label{sec:unique}
In this section, we will obtain the optimal convergence radius and
the uniqueness of the solution.

\begin{lemma} \label{pr:best}
If  $(\beta(\rho
f'(\rho)-f(\rho))+\sqrt{2}c\beta^2(f'(\rho)+1))/\rho(1-\beta(f'(\rho)+1))=1$
and $\rho < \kappa$, then  $r=\rho$ is the best  possible.
\end{lemma}
\begin{proof}
Define the function $ h:(-\kappa,
\,\kappa)\to \mathbb{R}$ by
\begin{equation} \label{eq:dh1}
 h(t)=
      \begin{cases}
       -t/\beta+ t-f(-t), \quad \;\; t\in  (-\kappa, \,0],\\
       -t/\beta+ t + f(t), \quad \quad   \; t\in [0, \,\kappa).
      \end{cases}
\end{equation}
It is straightforward to show that $h(0)=0$,  $h'(0)=-1/\beta,$
$h'(t)=-1/\beta+ 1 + f'(|t|)$ and that
$$
\left|h'(t)-h'(\tau t)\right| \leq
    f'(|t|)-f'(\tau|t|), \quad\tau \in [0,1], \quad t\in (-\kappa,\, \kappa).
$$
So, $F=h$ satisfy all assumption of Theorem \ref{th:nt} with
$c=|h(0)|=0$. Thus, as $\rho<\kappa $,  it  suffices to show
that the Gauss-Newton's method applied for solving \eqref{eq:p1},
with $F=h$ and starting point $x_0=\rho$ does not converges. Since
$c=0$ our assumption  becomes
\begin{equation} \label{eq:assrf}
(\beta(\rho
f'(\rho)-f(\rho))/\rho(1-\beta(f'(\rho)+1))=1.
\end{equation}
Hence the definition of $h$ in \eqref{eq:dh1} together with last equality yields
$$
x_{1} =\rho-\frac{h'(\rho)^{T}h(\rho)}{h'(\rho)^{T}h'(\rho)}=
\rho-\frac{-\rho/\beta+ \rho + f(\rho)}{-1/\beta +1+ f'(\rho)}=
-\rho\left(\frac{\beta(\rho
f'(\rho)-f(\rho))}{\rho(1-\beta(f'(\rho)+1))}\right)=-\rho .
$$
Again, definition of $h$ in \eqref{eq:dh1}  and assumption \eqref{eq:assrf} gives
$$
x_{2} =-\rho-\frac{h'(-\rho)^{T}h(-\rho)}{h'(-\rho)^{T}h'(-\rho)}=
-\rho-\frac{\rho/\beta -\rho -f(\rho)}{-1/\beta +1+ f'(\rho)}=
\rho\left(\frac{\beta(\rho
f'(\rho)-f(\rho))}{\rho(1-\beta(f'(\rho)+1))}\right)=\rho .
$$
Therefore, Gauss-Newton's method, for solving \eqref{eq:p1} with $F=h$ and staring point
$x_0=\rho$, produces the cycle
$$
x_0=\rho,\quad x_1=-\rho, \quad  x_2=\rho,\; \ldots \;,
$$
as a consequence, it does not converge. Therefore, the lemma is
proved.
\end{proof}

\begin{lemma} \label{pr:uniq}
 If additionally, {\bf h4} holds, then the point $x_*$  is  the unique solution  of \eqref{eq:p1} in $B(x_*, \sigma)$.
\end{lemma}
\begin{proof} Assume that $y \in B(x_*, \sigma),$ $y\neq x_*$ is also a solution of \eqref{eq:p1}.  Since $F'(y)^TF(y)=~0,$ we have
$$
y-x_*=y-x_*-[F'(x_*)^TF'(x_*)]^{-1}F'(y)^TF(y).
$$
Using $F'(x_*)^TF(x_*)=~0$,  after some algebraic manipulation the above equality becomes
\begin{multline*}
y-x_*=[F'(x_*)^TF'(x_*)]^{-1}F'(x_*)^T[F'(x_*)(y-x_*)-F(y)+F(x_*)]\\+[F'(x_*)^TF'(x_*)]^{-1}(F'(x_*)^T-F'(y)^T)F(y).
\end{multline*}
Combining the last equation with  properties of the norm and
definitions of $c$, $\beta$ and $\beta_0,$ we obtain
$$
\|y-x_*\|\leq\beta\int_{0}^{1}\|F'(x_*)-F'(x_*+u(y-x_*))\|\|y-x_*\|du+c\beta_{0}\|F'(x_*)^T-F'(y)^T\|.
$$
Using  \eqref{Hyp:MH} with $x=x_*+u(y-x_*)$ and $\tau=0$ in the first
term of the right-hand side, and $x=y$ and $\tau=0$ in the second term of the
right-hand side in last inequality, we have
$$
\|y-x_*\|\leq \beta
\int_{0}^{1}[f'(u\|y-x_*\|)-f'(0)]\|y-x_*\|du+c\beta_{0}[f'(\|y-x_*\|)-f'(0)].
$$
Evaluating the above integral and using ${\bf h1}$, the latter
inequality becomes
$$
\|y-x_*\|\leq
\left(\beta\left[\frac{f(\|y-x_*\|)}{\|y-x_*\|}+1\right]+c\beta_{0}\left[\frac{f'(\|y-x_*\|)+1}{\|y-x_*\|}\right]\right)\|y-x_*\|,
$$
Since $0<\|y-x_*\|< \sigma$, using Proposition~\ref{pr:sig} with $t=\|y-x_*\|$, we have
$
\|y-x_*\|< \|y-x_*\|,
$
which is a contradiction. Therefore,  $y=x_*$.
\end{proof}
\begin{remark}
Note that in the above lemma we  have used the fact that
condition \eqref{Hyp:MH} holds only for $\tau=0$.
\end{remark}
\subsection{Proof of {\bf Theorem \ref{th:nt}}} \label{sec:proof}
First of all, note that the  equation in \eqref{eq:DNS} together
\eqref{NF} imply that   the sequence $\{x_k\}$  satisfies
\begin{equation} \label{GF}
x_{k+1}=G_F(x_k),\qquad k=0,1,\ldots \,.
\end{equation}
\begin{proof}
Since  $x_0\in B(x_*,r)/\{x_*\},$ i.e., $0<\|x_k-x_*\|<r,$ by  combination of
Lemma~\ref{wdns}, last inequality in Lemma~\ref{l:wdef} and  induction argument it is easy to see that  $\{x_k\}$  is well defined and remains in $B(x_*,r)$.

Now, our goal is  to  show that  $\{x_k\}$ converges to $x_*$.
As, $\{x_k\}$ is well defined and  contained in  $B(x_*,r)$,
combining \eqref{GF} with Lemma \ref{l:wdef} we have
\begin{multline*}
\|x_{k+1}-x_{*}\|\leq
\frac{\beta[f'(\|x_k-x_*\|)\|x_k-x_*\|-f(\|x_k-x_*\|)]}{\|x_k-x_*\|^2[1-\beta(f'(\|x_k-x_*\|)+1)]}\|x_k-x_*\|^2
\\+\frac{\sqrt{2}c
\beta^2[f'(\|x_k-x_*\|)+1]}{\|x_k-x_*\|[1-\beta(f'(\|x_k-x_*\|)+1)]}\|x_k-x_*\|,
\end{multline*}
for all $ k=0,1,\ldots.$. Using again  \eqref{GF} and  the second part of and Lemma \ref{l:wdef}
it easy to conclude that
\begin{equation} \label{eq:icq2}
\|x_{k}-x_*\|< \|x_0-x_*\|, \qquad \;k=1, 2 \ldots.
\end{equation}
Hence combining two last inequalities  with last part of
Proposition~\ref{pr:incr101} we obtain that
\begin{multline*}
\|x_{k+1}-x_{*}\|\leq
\frac{\beta[f'(\|x_0-x_*\|)\|x_0-x_*\|-f(\|x_0-x_*\|)]}{\|x_0-x_*\|^2[1-\beta(f'(\|x_0-x_*\|)+1)]}\|x_k-x_*\|^2
\\+\frac{\sqrt{2}c
\beta^2[f'(\|x_0-x_*\|)+1]}{\|x_0-x_*\|[1-\beta(f'(\|x_0-x_*\|)+1)]}\|x_k-x_*\|,
\end{multline*}
for all $ k=0,1,\ldots$, which is the  inequality \eqref{eq:q2}.
Now, using \eqref{eq:icq2} and last inequality we have
$$
\|x_{k+1}-x_*\| \leq
    \left[\frac{\beta[f'(\|x_0-x_*\|)\|x_0-x_*\|-f(\|x_0-x_*\|)]+\sqrt{2}c \beta^2[f'(\|x_0-x_*\|)+1]}{\|x_0-x_*\|[1-\beta(f'(\|x_0-x_*\|)+1)]}\right]\|x_k-x_*\|,
$$
for all  $k=0,1,\ldots$. Applying  Proposition
\ref{pr:incr102} with $t=\|x_0-x_*\|$ it is straightforward  to
conclude from latter inequality  that $\{\|x_{k}-x_*\|\}$
converges to zero. So,  $\{x_k\}$ converges to $x_*$.  The optimal
convergence radius was proved in Lemma~\ref{pr:best}  and the last
statement of theorem was proved in Lemma \ref{pr:uniq}.
\end{proof}
\section{Special cases} \label{sec:ec}
In this section, we present two special cases of
Theorem~\ref{th:nt}. They include the classical convergence theorem on Gauss-Newton's method under  Lipschitz condition and Smale's theorem on Gauss-Newton for  analytical functions.

\subsection{Convergence result for Lipschitz condition}
In this section we show a correspondent theorem to Theorem
\ref{th:nt} under  Lipschitz condition (see
\cite{DN1} and  \cite{NC1} ) instead of the general assumption
\eqref{Hyp:MH}.

\begin{theorem}\label{th:ntqnnm}
Let $\Omega\subseteq \mathbb{R}^{n}$ be an open set, $F:{\Omega}\to \mathbb{R}^{m}$ be continuously differentiable in
$\Omega$ and $m\geq n$. Let  $x_* \in \Omega$   and
$$c:=\|F(x_*)\|,\qquad \beta:=\|[F'(x_*)^TF'(x_*)]^{-1}F'(x_*)^T\|, \qquad
\kappa:=\sup \left\{ t\in [0, R): B(x_*, t)\subset\Omega \right\}.
$$
Suppose that $x_*$ is a solution of \eqref{eq:p1}, $F '(x_*)$ has full rank
and there exists a $K>0$ such that
$$
\sqrt{2}c\beta^2 K<1, \qquad \qquad \left\|F'(x)-F'(y)\right\| \leq K\|x-y\|, \qquad \forall\; x,
y\in B(x_*, \kappa).$$ Let
$$r:=\min\left\{\kappa,\,\big(2-2\sqrt{2}K\beta^2c\big)/\big(3K\beta\big)\right\}.$$
Then, the Gauss-Newton methods for solving \eqref{eq:p1}, with
initial point $x_0\in B(x_*, r)/\{x_*\}$
$$
x_{k+1}={x_k}- \left[F'(x_k)^TF'(x_k)\right]^{-1}F'(x_k)^TF(x_k), \qquad  \; k=0,1,\ldots.
$$
is well defined,  the sequence generated $\{x_k\}$ is contained in $B(x_*,r),$  converges to $x_*$
and
$$   \|x_{k+1}-x_*\| \leq
    \frac{\beta K}{2(1-\beta K\|x_0-x_*\|)}\|x_k-x_*\|^2+\frac{\sqrt{2}c \beta^2K}{1-\beta K\|x_0-x_*\|}\|x_k-x_*\|, \qquad  \; k=0,1,\ldots.
$$
Moreover, if $(2-2\sqrt{2}K\beta^2c)/(3K\beta)<\kappa$, then
$r=(2-2\sqrt{2}K\beta^2c)/(3K\beta)$ is the best  possible
convergence radius.\\
 If, additionally, $2c\beta_0K<1,$ then the point $x_*$ is the unique solution of \eqref{eq:p1}
 in $B(x_*,(2-2c\beta_0K)/(\beta K))$, where $\beta_0:=\|[F'(x_*)^TF'(x_*)]^{-1}\|$.
\end{theorem}
 \begin{proof}
It is immediate to  prove that  $F$, $x_*$ and $f:[0,
\kappa)\to \mathbb{R}$ defined by $ f(t)=Kt^{2}/2-t, $ satisfy the
inequality \eqref{Hyp:MH}, conditions {\bf h1} and {\bf h2}. Since $\sqrt{2}c\beta^2 K<1$ and $2c\beta_0K<1$ the conditions {\bf h3} and {\bf h4}  also hold. In this case, it is easy to see that the constants
$\nu$  and $\rho$ as defined in Theorem~\ref{th:nt}, satisfy
$$0<\rho=(2-2\sqrt{2}K\beta^2c)/(3K\beta))\leq \nu=1/\beta K ,$$
as a consequence,
$
0<r=\min \{\kappa,\,\rho\}.
$
Moreover,  it is straightforward  to show that
$$
[\beta(\rho
f'(\rho)-f(\rho))+\sqrt{2}c\beta^2(f'(\rho)+1)]/[\rho(1-\beta(f'(\rho)+1))]=1,
$$
and  $[\beta(f(t)/t+1)+c\beta_0(f'(t)+1)/t]<
 1 $ for all $t \in (0,(2-2c\beta_0K)/(\beta K)).$
Therefore, as  $F$, $r$,  $f$ and $x_*$ satisfy all of the
hypotheses of  Theorem \ref{th:nt}, taking  $x_0\in B(x_*,
r)\backslash \{x_*\}$ the statements of the theorem follow from
Theorem~\ref{th:nt}.
\end{proof}
For  the zero-residual problems, i.e., $c=0$,  the Theorem \ref{th:ntqnnm} becomes:
\begin{corollary}\label{cor:li1}
Let $\Omega\subseteq \mathbb{R}^{n}$ be an open set, $F:{\Omega}\to \mathbb{R}^{m}$ be continuously differentiable in
$\Omega$ and $m\geq n$. Let  $x_* \in \Omega$ and
$$
\beta:=\|[F'(x_*)^TF'(x_*)]^{-1}F'(x_*)^T\|, \qquad \kappa:=\sup \left\{ t\in [0, R): B(x_*, t)\subset\Omega \right\}.
$$
Suppose that  $F(x_*)=0,$ $F '(x_*)$ has full rank  and there
exists a $K>0$ such that
$$\left\|F'(x)-F'(y)\right\| \leq K\|x-y\|, \qquad \forall\; x,
y\in B(x_*, \kappa).$$ Let
$$r:=\min\left\{\kappa,\,2/(3K\beta)\right\}.$$
Then, the Gauss-Newton methods for solving \eqref{eq:p1}, with
initial point $x_0\in B(x_*, r)/\{x_*\}$
$$x_{k+1}={x_k}- \left[F'(x_k)^TF'(x_k)\right]^{-1}F'(x_k)^TF(x_k), \qquad \; k=0,1,\ldots,
$$
is well defined,  the sequence generated $\{x_k\}$ is contained in $B(x_*,r)$ and converges to $x_*$
which is  the unique solution of \eqref{eq:p1}  in $B(x_*,2/(\beta
K))$. Moreover, there holds
 $$   \|x_{k+1}-x_*\| \leq
    \frac{\beta K}{2(1-\beta K\|x_0-x_*\|)}\|x_k-x_*\|^2, \qquad \; k=0,1,\ldots.
$$
If, additionally $2/(3K\beta)<\kappa$, then $r=2/(3K\beta)$ is the
best possible convergence radius.\\
 \end{corollary}

\begin{remark}
When $m=n,$ the Corollary \ref{cor:li1} merge in the
results on the Newton's method for sol\-ving nonlinear equations
$F(x)=0$, which has been obtained by
Ferreira \cite{F08} in Theorem 3.1 and Remark~3.3.
\end{remark}

\subsection{Convergence result under Smale's condition }
In this section we present a  correspondent theorem to Theorem
\ref{th:nt} under  Smale's condition. For more details about
 Smale's condition   see \cite{S86}.

\begin{theorem}\label{theo:Smale}
Let $\Omega\subseteq \mathbb{R}^{n}$ be an open set,
$F:{\Omega}\to \mathbb{R}^{m}$ an analytic function and $m\geq n$. Let $x_* \in
\Omega$ and
$$c:=\|F(x_*)\|, \qquad \beta:=\|[F'(x_*)^TF'(x_*)]^{-1}F'(x_*)^T\|, \qquad \kappa:=\sup\{t
>0 : B(x_*, t)\subset \Omega\}.$$
Suppose that $x_*$ is a solution of \eqref{eq:p1},  $F '(x_*)$ has full
rank and
\begin{equation} \label{eq:SmaleCond}
 \qquad \gamma := \sup _{ n > 1 }\left\| \frac
{F^{(n)}(x_*)}{n !}\right\|^{1/(n-1)}<+\infty,\qquad
\qquad 2\sqrt{2}c\beta^2\gamma< 1.
\end{equation}
 Let
$a:=(2+3\beta-\sqrt{2}c\beta^2\gamma)$,
$b:=4(1+\beta)(1-2\sqrt{2}c\beta^2\gamma)$ and
$$
r:=\min \left\{\kappa, \big( a-
\sqrt{a^2-b}\big)/\big(2\gamma(1+\beta)\big)\right\}.
$$
Then, the Gauss-Newton methods for solving \eqref{eq:p1}, with
initial point $x_0\in B(x_*, r)/\{x_*\}$
$$
x_{k+1}={x_k}- \left[F'(x_k)^TF'(x_k)\right]^{-1}F'(x_k)^TF(x_k),\qquad \; k=0,1,\ldots,
$$
is well defined, the sequence generated $\{x_k\}$ is contained in $B(x_*,r)$, converges to $x_*$
and
\begin{align*}
    \|x_{k+1}-x_*\| \leq &
    \frac{\beta\gamma}{(1-\gamma \|x_0-x_*\|)^2-\beta\gamma(2\|x_0-x_*\|-\gamma\|x_0-x_*\|^2)}\|x_k-x_*\|^2\\+
    &\frac{\sqrt{2}c \beta^2\gamma(2-\gamma\|x_0-x_*\|)}{(1-\gamma \|x_0-x_*\|)^2-\beta\gamma(2\|x_0-x_*\|-\gamma\|x_0-x_*\|^2)}\|x_k-x_*\|,\quad \; k=0,1,\ldots.
\end{align*}
Moreover, if $(a- \sqrt{a^2-b})/(2\gamma(1+\beta))<\kappa$, then
$r=(a- \sqrt{a^2-b})/(2\gamma(1+\beta))$ is the best possible
convergence
  radius.\\
If additionally, $4c\beta_0\gamma<1,$ then the point $x_*$ is
the unique solution \eqref{eq:p1} in
$B(x_*,\sigma)$, where $\sigma:=(\omega_1-\sqrt{\omega_1^2-\omega_2})/(2\gamma(1+\beta)),$ $\omega_1:=(2+\beta-c\beta_0), \;
\omega_2:=4(1+\beta)(1-2c\beta_0\gamma), \; \beta_0:=\|[F'(x_*)^TF'(x_*)]\|.$
 \end{theorem}

We need the following result to prove the above theorem.
\begin{lemma} \label{lemma:qc1}
Let $\Omega\subseteq \mathbb{R}^{n}$ be an open set and
$F:{\Omega}\to \mathbb{R}^{m}$ an analytic function.  Suppose that
$x_*\in \Omega$  and  $B(x_{*},
1/\gamma) \subset \Omega$, where $\gamma$ is defined in
\eqref{eq:SmaleCond}. Then, for all $x\in B(x_{*}, 1/\gamma)$
there holds
$$
\|F''(x)\| \leqslant  (2\gamma)/( 1- \gamma \|x-x_*\|)^3.
$$
\end{lemma}
\begin{proof}
Let $x\in \Omega$.  Since $F$ is an analytic function, we have
$$
F''(x)= \sum _{n=0}^{\infty}\frac {1}{n!}F^{(n + 2)}(x_{*})(x -
x_{*})^{n}.
$$
Combining  \eqref{eq:SmaleCond} and  the above equation we obtain,
after some simple calculus, that
$$
\|F''(x)\| \leqslant \,\gamma \sum
_{n=0}^{\infty}(n+2)(n+1)(\gamma ||x-x_{*}||)^{n} .
$$
On the other hand, as $B(x_{*}, 1/\gamma) \subset \Omega$ we have
 $\gamma \|x-x_*\|< 1$. So,  from Proposition~\ref{le:ess}  we
conclude
$$
\frac{2}{(1-\gamma\|x-x_*\|)^3}=\sum_{n=0}^{\infty}(n+2)(n+1)(\gamma
||x-x_{*}||)^{n}.
$$
Combining the two above equations, we obtain the desired result.
\end{proof}
The next result gives a condition that is easier to check than
condition \eqref{Hyp:MH}, whenever the functions under
consideration are twice continuously differentiable.
\begin{lemma} \label{lc}
Let $\Omega\subseteq \mathbb{R}^{n}$ be an open set, $x_*\in \Omega$  and
  $F:{\Omega}\to \mathbb{R}^{m}$ be twice continuously on $\Omega$. If there exists a \mbox{$f:[0,R)\to \mathbb {R}$} twice continuously differentiable such that
 \begin{equation} \label{eq:lc2}
\|F''(x)\|\leqslant f''(\|x-x_*\|),
\end{equation}
for all $x\in  \Omega$ such that  $\|x-x_*\|<R$. Then $F$ and $f$
satisfy \eqref{Hyp:MH}.
\end{lemma}
\begin{proof}
   Taking $\tau \in [0,1]$ and $x\in \Omega$, such that $x_*+\tau(x-x_*)\in \Omega$  and   $\|x-x_*\|<R$, we
  obtain that
 $$
 \|\left[F'(x)-F'(x_*+\tau(x-x_*))\right]\|\leq
  \int_{\tau}^{1}\|F''(x_*+t(x-x_*))\|\,\|x-x_*\|dt.
$$
  Now, as $\|x-x_*\|<R$ and $f$ satisfies \eqref{eq:lc2}, we obtain from the last inequality that
\begin{align*}
 \|\left[F'(x)-F'(x_*+\tau(x-x_*))\right]\|&\leq \int_{\tau}^{1}f''(t\|x-x_*\|)\|x-x_*\|dt.
\end{align*}
Evaluating the latter integral, the statement follows.
\end{proof}
{\bf [Proof of Theorem \ref{theo:Smale}]}.  Consider the real
function $f:[0,1/\gamma) \to \mathbb{R}$ defined by
$$
f(t)=\frac{t}{1-\gamma t}-2t.
$$
It is straightforward to show that $f$ is  analytic and that
$$
f(0)=0, \quad f'(t)=1/(1-\gamma t)^2-2, \quad f'(0)=-1, \quad
f''(t)=(2\gamma)/(1-\gamma t)^3, \quad f^{n}(0)=n!\,\gamma^{n-1},
$$
for $n\geq 2$. It follows from the last equalities
that $f$ satisfies {\bf h1}  and  {\bf h2}. Since $ 2\sqrt{2}c\beta^2\gamma< 1$ and  $4c\beta_0\gamma<1$ the  conditions  {\bf h3}  and  {\bf h4} also hold. Now, as
$f''(t)=(2\gamma)/(1-\gamma t)^3$ combining Lemmas ~\ref{lc},
\ref{lemma:qc1}  we conclude that $F$  and $f$ satisfy
\eqref{Hyp:MH} with $R=1/\gamma$. In this case, it is easy to see
that the constants $\nu$  and $\rho$ as defined in Theorem \ref{th:nt}, satisfy
$$
0<\rho=(a-\sqrt{a^2-b})/(2\gamma(1+\beta))<\nu=((1+\beta)-\sqrt{\beta(1+\beta)})/(\gamma(1+\beta))<1\gamma,
$$
and as a consequence,
$
0<r=\min \{\kappa,\rho\}.
$
Moreover,   it is not hard to see that
$$
[\beta(\rho
f'(\rho)-f(\rho))+\sqrt{2}c\beta^2(f'(\rho)+1)]/[\rho(1-\beta(f'(\rho)+1))]=1,
$$
and $[ \beta(f(t)/t+1)+c\beta_0(f'(t)+1)/t]<1$ for all $t \in (0,\sigma$).
Therefore, as  $F$,
$\sigma$, $f$ and $x_*$ satisfy all hypothesis of  Theorem
\ref{th:nt}, taking $x_0\in B(x_*, r)\backslash \{x_*\}$, the
statements of the theorem follow from Theorem \ref{th:nt}.\qed

For the zero-residual problems, i.e., $c=0$,  the Theorem \ref{theo:Smale} becomes:
\begin{corollary}\label{cor:tt}
Let $\Omega\subseteq \mathbb{R}^{n}$ be an open set,
$F:{\Omega}\to \mathbb{R}^{m}$ an analytic function and $m\geq n$. Let $x_* \in
\Omega,$   and
$$\beta:=\|[F'(x_*)^TF'(x_*)]^{-1}F'(x_*)^T\|, \qquad \kappa:=\sup\{t
>0 : B(x_*, t)\subset \Omega\}.$$
Suppose that $F(x_*)=0$,  $F '(x_*)$ has full
rank and
$$
 \gamma := \sup _{ n > 1 }\left\| \frac
{F^{(n)}(x_*)}{n !}\right\|^{1/(n-1)}<+\infty.
$$
 Let
$$
r:=\min \left\{\kappa,\big( 2+3\beta-
\sqrt{\beta(8+9\beta)}\big)/\big(2\gamma(1+\beta)\big)\right\}.
$$
Then, the Gauss-Newton methods for solving \eqref{eq:p1}, with
initial point $x_0\in B(x_*, r)/\{x_*\}$
$$
x_{k+1}={x_k}- \left[F'(x_k)^TF'(x_k)\right]^{-1}F'(x_k)^TF(x_k),\qquad  k=0,1,\ldots,
$$
is well defined, is contained in $B(x_*,r)$ and converges to $x_*$
which   is the unique solution of \eqref{eq:p1} in
$B(x_*,1/(\gamma(1+\beta)))$. Moreover, there holds
$$
    \|x_{k+1}-x_*\| \leq
    \frac{\beta\gamma}{(1-\gamma \|x_0-x_*\|)^2-\beta\gamma(2\|x_0-x_*\|-\gamma\|x_0-x_*\|^2)}\|x_k-x_*\|^2,\qquad  k=0,1,\ldots.
$$
If, additionally, $(2+3\beta-
\sqrt{\beta(8+9\beta)})/(2\gamma(1+\beta))<\kappa$, then
$r=((2+3\beta- \sqrt{\beta(8+9\beta)})/(2\gamma(1+\beta))$ is the
best possible convergence
  radius.
 \end{corollary}

\begin{remark}
When $m=n$, the Corollary \ref{cor:tt}  is similar to the
results on the Newton's method for solving nonlinear equations
$F(x)=0$, which has been obtained by Ferreira
\cite{F08} in Theorem 3.4.
\end{remark}


\def\cprime{$'$}

\end{document}